\documentclass{amsart}
\usepackage{amssymb}
\usepackage{amsmath}
\usepackage{latexsym}
\usepackage{graphicx} 
\usepackage{amscd}
\usepackage{eufrak}
\usepackage{mathrsfs}
\usepackage[english]{babel}
\usepackage{hyperref}
\usepackage{verbatim}
\usepackage{tikz,tikz-cd}
\usetikzlibrary{matrix,arrows, babel}

\newtheorem{theorem}[equation]{Theorem}
\newtheorem{theorem-definition}[equation]{Theorem-Definition}
\newtheorem{lemma-definition}[equation]{Lemma-Definition}
\newtheorem{definition-prop}[equation]{Proposition-Definition}
\newtheorem{corollary}[equation]{Corollary}
\newtheorem{prop}[equation]{Proposition}
\newtheorem*{prop*}{Proposition}
\newtheorem{lemma}[equation]{Lemma}

\theoremstyle{definition}

\newtheorem{definition}[equation]{Definition}

\newcommand{\llpar}{(\negthinspace(}
\newcommand{\rrpar}{)\negthinspace)}

\theoremstyle{definition}
\newtheorem{example}[equation]{Example}

\newtheorem{remark}[equation]{Remark}

\newcommand{\LL}{\ensuremath{\mathbb{L}}}

\newcommand{\Z}{\ensuremath{\mathbb{Z}}}
\newcommand{\Q}{\ensuremath{\mathbb{Q}}}

\newcommand{\R}{\ensuremath{\mathbb{R}}}
\newcommand{\C}{\ensuremath{\mathbb{C}}}

\newcommand{\A}{\ensuremath{\mathbb{A}}}

\newcommand{\cY}{\ensuremath{\mathscr{Y}}}

\renewcommand{\R}{\ensuremath{\mathbb{R}}}
\renewcommand{\C}{\ensuremath{\mathbb{C}}}

\renewcommand{\A}{\ensuremath{\mathbb{A}}}

\renewcommand{\cY}{\ensuremath{\mathscr{Y}}}

\newcommand{\Spec}{\ensuremath{\mathrm{Spec}\,}}

\newcommand{\red}{\mathrm{red}}

\newcommand{\Var}{\mathrm{Var}}

\newcommand{\Hom}{\mathrm{Hom}}

\newcommand{\trop}{\mathrm{trop}}

\newcommand{\VF}{\mathrm{VF}}

\numberwithin{equation}{subsection}

\begin{document}
\title[Grothendieck rings of polytopes]{Grothendieck rings of polytopes and non-archimedean semi-algebraic sets}

\author[Johannes Nicaise]{Johannes Nicaise}
\address{ Imperial College,
Department of Mathematics, South Kensington Campus,
London SW7 2AZ, UK, and KU Leuven, Department of Mathematics, Celestijnenlaan 200B, 3001 Heverlee, Belgium} \email{johannes.nicaise@kuleuven.be}

\begin{abstract}
Let $\Gamma$ be a divisible subgroup of $(\R,+)$. Our central result  states that, at the level of Grothendieck groups, the classification of $\Gamma$-rational polyhedra in $\R^n$ up to affine transformations in  $\Gamma^n\rtimes \mathrm{GL}_n(\Z)$ is equivalent to the classification up to affine transformations in  $\Gamma^n\rtimes \mathrm{GL}_n(\Q)$. We prove this by giving an explicit description of these Grothendieck groups. This yields, in particular, a positive answer to the basic case of a question by Hrushovski and Kazhdan; all other cases are still open. As a second application, we give a simple description of the kernel of the motivic volume for non-archimedean semi-algebraic sets, which is a key ingredient of Hrushovski and Kazhdan's theory of motivic integration.	
\end{abstract}

\maketitle

\section{Introduction}\label{sec:intro}
 The central problem of this paper is the classification of rational polyhedra in $\R^n$ up to cut-and-paste relations. More precisely, let us fix a divisible subgroup $\Gamma$ of $(\R,+)$. A {\em $\Gamma$-constructible set} in $\R^n$ is a finite Boolean combination of closed half-spaces in $\R^n$ defined by equations of the form 
 $$a_1x_1+\ldots+ a_nx_n\geq c$$
 where $a_1,\ldots,a_n$ are integers and $c$ lies in $\Gamma$. A special class of $\Gamma$-constructible sets are the {\em $\Gamma$-rational polyhedra}, finite intersections of closed half-spaces as above. 

 There are two natural ways to define isomorphisms between $\Gamma$-constructible sets in $\R^n$. We can consider all affine transformations of $\R^n$ whose linear part lies in $\mathrm{GL}_n(\Q)$ and whose translation vector lies in $\Gamma^n$. Alternatively, we can further restrict to affine transformations with linear part in $\mathrm{GL}_n(\Z)$. 
 These two types of affine transformations both preserve the class of $\Gamma$-constructible sets in $\R^n$.
  We say that two $\Gamma$-constructible sets in $\R^n$ are {\em rationally isomorphic} over $\Gamma$ if they are connected by an affine transformation of the first kind, and {\em integrally isomorphic} over $\Gamma$ if they are connected by an affine transformation of the second kind.
  
  It is easy to construct examples of $\Gamma$-constructible sets that are rationally isomorphic but not integrally isomorphic; a simple example are the intervals $[0,1]$ and $[0,2]$ (with $\Gamma=\Q$). Surprisingly, this distinction disappears at the level of the Grothendieck groups $\mathbf{K}(\Gamma[n])$ and $\mathbf{K}_{\Q}(\Gamma[n])$, the free abelian groups on integral, resp.~rational, isomorphism classes of $\Gamma$-constructible sets in $\R^n$ modulo the {\em scissor relations} that allow to decompose a $\Gamma$-constructible set into finitely many constructible pieces. We refer to Section \ref{ss:GrGrPol} for precise definitions. Our central result (Theorem \ref{theo:explicit}) states that the natural group morphism 
  $$\mathbf{K}(\Gamma[n])\to \mathbf{K}_{\Q}(\Gamma[n])$$
  is an isomorphism. The divisibility condition on $\Gamma$ is essential: we will show in Proposition \ref{prop:distinct} that  $\mathbf{K}(\Gamma[1])\to \mathbf{K}_{\Q}(\Gamma[1])$ is injective only when $\Gamma$ is divisible.
    We prove Theorem \ref{theo:explicit} by giving an explicit description of $\mathbf{K}(\Gamma[n])$, which shows that the class of a $\Gamma$-constructible set in $\R^n$ is completely determined by its Euler characteristic and so-called {\em bounded} Euler characteristic (this was already known for the groups $\mathbf{K}_{\Q}(\Gamma[n])$). Our proof  rests upon a generalization of the Brianchon-Gram decomposition theorem (Theorem \ref{theo:ubBriGra}).

  The broader question about the difference between $\mathbf{K}(\Gamma[n])$ and $\mathbf{K}_{\Q}(\Gamma[n])$ for arbitrary subgroups $\Gamma$ of $\R$ (and in even greater generality) was raised by Hrushovski and Kazhdan in their seminal paper on motivic integration over non-archimedean valued fields \cite{HK}. They formulated a precise expectation in Question 9.9 of \cite{HK}. In Section \ref{ss:question}, we show how our main result in Theorem \ref{theo:explicit} confirms this expectation when $\Gamma$ is divisible. This is the most basic case, but also the only one where an answer is currently known.

 In Section \ref{sec:appli}, we present the application that motivated our study of the groups $\mathbf{K}(\Gamma[n])$. Let $(K,v)$ be an algebraically closed $\R$-valued field with residue field $k$ of characteristic $0$. Let $\Gamma=v(K^{\ast})\subset \R$ be the value group.
  A good example to keep in mind is the case where $K$ is the field of complex Puiseux series; then $k=\C$ and $\Gamma=\Q$. The main goal of Hrushovski and Kazhdan's theory of motivic integration in \cite{HK} is the classification of {\em semi-algebraic} sets over $K$, that is, subsets of $K^n$ defined by inequalities between valuations of polynomials over $K$ (see Section \ref{ss:GrSemAlg}). They achieve this by expressing the Grothendieck ring of semi-algebraic sets over $K$ in terms of the graded ring $$\mathbf{K}^{\dim}(\Gamma)=\bigoplus_{n\geq 0}\mathbf{K}(\Gamma[n])$$ and the Grothendieck ring of algebraic $k$-varieties graded by dimension; we recall their result in Theorem \ref{theo:HKorig}. An important application of this expression is the construction of the {\em motivic volume}, a ring morphism from the Grothendieck ring of semi-algebraic sets over $K$ to the Grothendieck ring of algebraic varieties over $k$. The motivic volume measures the behaviour of semi-algebraic sets under reduction to the residue field $k$. It plays a key role in several applications of the theory; see in particular \cite{NPS,NP,NiSh}. In Theorem \ref{theo:kernel}, we deduce from Theorem \ref{theo:explicit} a simple description of the kernel of the motivic volume, which explains precisely which part of the geometry is collapsed under reduction from $K$ to $k$: we are formally contracting the closed unit ball in $K$ to a point.

\subsection*{Acknowledgements} The inspiration for this article was a lecture series on Hrushovski-Kazhdan motivic integration that I gave as a John von Neumann Visiting Professor at TU Munich in the autumn term of 2022. I am grateful to TU Munich and my host, Christian Liedtke, for their hospitality and support. I am also endebted to the students in the course for their thoughtful questions, which helped to shape this article. Further thanks go out to Raf Cluckers for clarifying the state of the literature in the final stages of the project.

The author was supported
by EPSRC grant EP/S025839/1, grants G079218N and G0B17121N of the
Fund for Scientific Research--Flanders, and long term structural funding (Methusalem grant) by the Flemish Government.

\section{Grothendieck rings of polyhedra}\label{sec:GroPol}
\subsection{$\Gamma$-rational polyhedra}\label{ss:GratPol}
 Let $\Gamma$ be a subgroup of $\R$. For every lattice $L$, we will write $L_{\Q}=L\otimes_{\Z}\Q$ and $L_{\R}=L\otimes_{\Z}\R$. 
   Let $N$ be a lattice of rank $n$, and denote by $M=\mathrm{Hom}_{\mathbb{Z}}(N,\mathbb{Z})$ the dual lattice.
  We identify $M_{\R}$ with the $\R$-vector space $\Hom_{\Z}(N,\R)$.  Our choice of the notations $N$ and $M$ is motivated by applications to tropical and non-archimedean geometry  which will not appear explicitly in this article ($M$ is then the character lattice of an algebraic torus). 

\begin{definition}\label{defi:polytope}
	A {\em $\Gamma$-rational polyhedron} in $N_{\R}$ is a finite intersection of closed half-spaces of the form 
	$$H=\{w\in N_{\R}\,|\,\ell(w)\geq c\}$$ where $\ell\in M\setminus \{0\}$ and $c\in \Gamma$.
 A {\em $\Gamma$-rational polytope} is a bounded $\Gamma$-rational polyhedron.		
  A {\em $\Gamma$-constructible set} in $N_{\R}$ is a finite Boolean combination of $\Gamma$-rational polyhedra (equivalently, of closed half-spaces $H$ as above).
\end{definition}

 It is obvious that every face of a $\Gamma$-rational polyhedron is again a $\Gamma$-rational polyhedron.  
 The definition of a $\Gamma$-rational polyhedron depends on the lattice $N$ and the divisible hull $\Gamma \otimes_{\Z}\mathbb{Q}$ of $\Gamma$ in $\R$, but not on $\Gamma$ itself, since we can always rescale the inequality in the formula for the half-space $H$. When considering $\Gamma$-rational polyhedra in $\R^n$, we will tacitly assume that $N$ is the standard lattice $\Z^n$, unless explicitly stated otherwise. 
 
 \begin{example}\item
 	\begin{enumerate}
 	\item The whole space $N_{\R}$ is a $\Gamma$-rational polyhedron in $N_{\R}$ (intersection of the empty family of half-spaces). The empty subset of $N_{\R}$ is a $\Gamma$-rational polytope if and only if $\Gamma\neq \{0\}$.  
 	\item For every finite subset $S$ of $N_{\Q}$, the convex hull of $S$ in $N_{\R}$ is a $\Q$-rational polytope.
 	\item The $\{0\}$-rational polyhedra in $N_{\R}$ are the rational polyhedral cones in $N_{\R}$. 
 	\item The half-space $\{(x,y)\in \R^2\,|\,y-\sqrt{2} x\geq 0\}$ is not a $\Gamma$-rational polyhedron in $\R^2$ for any choice of $\Gamma$.
 	\item The interval
 	$[\sqrt{2},\infty)$ is not a $\Q$-rational polyhedron in $\R$, but it is a $(\Q+\sqrt{2}\Q)$-rational polyhedron. 	
 	\item The semi-open interval $[0,1)$ is not a $\Q$-rational polyhedron in $\R$, but it is a $\Q$-constructible set.
 	\end{enumerate}	
 \end{example}

Let $N$ and $N'$ be lattices of the same rank $n$. Let $C$ and $C'$ be $\Gamma$-constructible subsets of $N_{\R}$, resp.~$N'_{\R}$. We say that $C$ and $C'$ are {\em rationally isomorphic over $\Gamma$} if there exist an isomorphism of $\Q$-vector spaces 
$\varphi\colon N_{\Q}\to N'_{\Q}$ and an element $b$ in $N'_{\Q}\otimes_{\Z}\Gamma$ such that the affine transformation
$$\psi\colon N_{\R}\to N'_{\R},\,w\mapsto \varphi_{\R}(w)+b$$ satisfies $\psi(C)=C'$. Here $\varphi_{\R}$ denotes the $\R$-linear extension of $\varphi$.
  We say that $C$ and $C'$ are {\em integrally isomorphic over $\Gamma$} if we can find $\varphi$ and $b$ such that $\varphi(N)=N'$ and $b\in N'\otimes_{\Z}\Gamma$. Note that we do not allow isomorphisms between $\Gamma$-constructible sets that live in vector spaces of different dimensions. When $\Gamma$ is clear from the context, we will simply speak of rational and integral isomorphisms between  $\Gamma$-constructible sets. The definition of an integral isomorphism is sensitive to $\Gamma$, and not only to its divisible hull $\Gamma\otimes_{\Z}\Q$.

\begin{example}\label{exam:notiso}
	\item 
	\begin{enumerate}
		\item The $\Z$-rational polytopes $[0,1]$ and $[0,2]$ in $\R$	are rationally isomorphic, but not integrally isomorphic: integral isomorphisms preserve the volumes of polytopes.
		\item The $\Z$-rational polytopes $\{0\}$ and $\{1/2\}$ in $\R$ are integrally isomorphic over $\Q$, but not over $\Z$: integral isomorphisms over $\Z$ preserve the lattice. 
		\end{enumerate}
\end{example}

Surprisingly, if $\Gamma$ is divisible, then the distinction between integral and rational isomorphisms disappears when we pass to the level of Grothendieck rings.

\subsection{Graded Grothendieck rings of $\Gamma$-rational polyhedra}\label{ss:GrGrPol}
 For every $n\geq 0$, we denote by $\mathbf{K}(\Gamma[n])$ the abelian group defined by the following presentation:
\begin{itemize}
	\item generators: integral isomorphism classes $[C]$ of $\Gamma$-constructible subsets in $\R^n$, with respect to the standard lattice $\Z^n$;
	\item relations: if $D\subset C$ are $\Gamma$-constructible subsets in $\R^n$, then 
	$$[C]=[D]+[C\setminus D].$$
\end{itemize}
These relations are called {\em scissor relations}, because they allow to decompose $\Gamma$-constructible sets into constructible pieces. 
 They imply, in particular, that $[\emptyset]=0$ in $\mathbf{K}(\Gamma[n])$. 
 We also consider the graded abelian group
$$\mathbf{K}^{\dim}(\Gamma)=\bigoplus_{n\geq 0}\mathbf{K}(\Gamma[n]).$$
If $C$ is a $\Gamma$-constructible subset of $\R^n$ and we want to make the grading explicit, then we write $[C]_n$ for the class of $C$ in $\mathbf{K}^{\dim}(\Gamma)$.

If $C$ is a $\Gamma$-constructible subset of $\R^m$ and $C'$ is a $\Gamma$-constructible subset of $\R^{n}$, then it follows immediately from the definition that $C\times C'$ is a $\Gamma$-constructible subset of $\R^{m+n}$. There exists a unique graded ring structure on $\mathbf{K}^{\dim}(\Gamma)$ such that 
$$[C]_m\cdot [C']_{n}=[C\times C']_{m+n}$$ for all such $C$ and $C'$. The identity element for the ring multiplication is $1=[\mathbf{pt}]_0$, the class of the unique point in $\R^0$.  We call this graded ring the {\em Grothendieck ring of $\Gamma$-constructible sets graded by dimension}. Note that the grading keeps track of the dimension of the ambient vector space $\R^n$ and does not refer to the dimension of the $\Gamma$-constructible subset $C$. We set 
$\sigma=[\,\{0\}\,]_1$. Multiplication by $\sigma$ shifts the grading by $1$: for every $\Gamma$-constructible set $C$ in $\R^n$, it sends the class of $C$ in $\mathbf{K}(\Gamma[n])$ to the class of $C\times \{0\}\subset \R^{n+1}$ in $\mathbf{K}(\Gamma[n+1])$.

 If $N$ is a lattice of rank $n$, then every $\Gamma$-constructible subset $C$  of $N_{\R}$ defines a unique element $[C]$ in $\mathbf{K}(\Gamma[n])$: 
 the choice of an isomorphism $\varphi\colon N\to \Z^n$ gives rise to a $\Gamma$-constructible subset $C'=\varphi_{\R}(C)$ in $\R^n$ that is integrally isomorphic to $C$, and the class $[C']$ is independent of the choice of $\varphi$, so that we may define $[C]=[C']$. The following basic calculation will be important for our analysis of the ring $\mathbf{K}^{\dim}(\Gamma)$.
 
 \begin{example}\label{exam:prodzero}
 Consider the $\Gamma$-constructible set	
 $C=\R_{> 0}\times \R_{\geq 0}$ in $\R^2$. We can partition $C$ into the $\Gamma$-constructible subsets
 $$\begin{array}{lcl}
 C_1&=&\{(x,y)\in \R^2\,|\,x > y\geq 0\},
 \\C_2&=&\{(x,y)\in \R^2\,|\,y \geq x > 0\}.	
 	\end{array}$$
 Each of these subsets is integrally isomorphic to $C$ {\em via} a linear transformation in $\mathrm{GL}_2(\Z)$, so that 
 $$[C]_2=[C_1]_2+[C_2]_2=[C]_2+[C]_2$$ in 
 $\mathbf{K}^{\dim}(\Gamma)$. Consequently,  
 $$[\R_{\geq 0}]_1\cdot [\R_{>0}]_1=[C]_2=0.$$
 \end{example}

 We similarly define abelian groups $\mathbf{K}_{\Q}(\Gamma[n])$ and a graded ring $\mathbf{K}_{\Q}^{\dim}(\Gamma)$ by replacing integral isomorphism classes by rational isomorphism classes as generators in the definition of the Grothendieck group. For every subgroup $\Gamma'$ of $\R$ that contains $\Gamma$, there are obvious ``extension of scalars'' morphisms of graded rings 
 $$\mathbf{K}^{\dim}(\Gamma)\to \mathbf{K}^{\dim}(\Gamma'),\qquad \mathbf{K}^{\dim}_{\Q}(\Gamma)\to \mathbf{K}^{\dim}_{\Q}(\Gamma')$$ that view $\Gamma$-constructible sets as a $\Gamma'$-constructible sets.
  There is also a natural surjective morphism of graded rings  
\begin{equation}\label{eq:compar}
	\mathbf{K}^{\dim}(\Gamma)\to \mathbf{K}^{\dim}_{\Q}(\Gamma)
\end{equation} 
that maps, for every integer $n\geq 0$ and every $\Gamma$-constructible subset $C$ of $\R^n$, the class of $C$ in $\mathbf{K}(\Gamma[n])$ to the class of $C$ in  $\mathbf{K}_{\Q}(\Gamma[n])$.

Assuming that $\Gamma$ is divisible, we will give an explicit description of the ring $\mathbf{K}^{\dim}(\Gamma)$ in Theorem \ref{theo:explicit}. It implies that, remarkably, the morphism \eqref{eq:compar} is an isomorphism, so that   we do not lose any information in the Grothendieck ring by passing from integral isomorphism classes to rational isomorphism classes. The following toy example illustrates this phenomenon.

\begin{example} The $\Z$-rational polytopes $C_1=[0,1]$ and $C_2=[0,2]$ in $\R$ are rationally isomorphic, but not integrally isomorphic (see Example \ref{exam:notiso}). Still, they define the same class in $\mathbf{K}(\Z[1])$: by the scissor relations, 
$[C_2]-[C_1]$ is equal to the class of $(1,2]$, which we can write as the difference of the classes of $(1,\infty)$ and $(2,\infty)$. These two classes are equal because 	$(1,\infty)$ and $(2,\infty)$ are integrally isomorphic {\em via} translation by $1$.
\end{example}

\section{Rational {\em vs.} integral isomorphisms}\label{sec:RatVsInt}
\subsection{Euler characteristics on $\mathbf{K}^{\dim}_{\Q}(\Gamma)$.}\label{ss:euler}
In order to arrive at an explicit description of $\mathbf{K}^{\dim}(\Gamma)$, we  
 recall two types of Euler characteristics that can be defined on the ring $\mathbf{K}^{\dim}_{\Q}(\Gamma)$; see Section \ref{ss:question} below for further background.  

The Euler characteristic 
$$\chi\colon \mathbf{K}^{\dim}_{\Q}(\Gamma)\to \Z$$
is the unique ring morphism that satisfies the following properties:
\begin{enumerate}
	\item  for every integer $n\geq 0$ and every non-empty $\Gamma$-rational polytope $P$ in $\R^n$, we have 
	$\chi([P]_n)=1$;
	\item we have $\chi([\R_{\geq 0}]_1)=0$.
\end{enumerate}
The {\em bounded} Euler characteristic 
$$\chi_b\colon \mathbf{K}^{\dim}_{\Q}(\Gamma)\to \Z$$
is the unique ring morphism such that, 
for every integer $n\geq 0$ and every non-empty $\Gamma$-rational polyhedron $P$ in $\R^n$, we have 
$\chi([P]_n)=1$. For any $\Gamma$-constructible subset $C$ in $\R^n$, we will also write $\chi(C)$ and $\chi_b(C)$ instead of $\chi([C]_n)$ and $\chi_b([C]_n)$.

The Euler characteristic $\chi$ can be defined by means of cell decomposition for $o$-minimal structures \cite{vdD}; in our setting, it agrees with the compactly supported topological Euler characteristic. The bounded Euler characteristic $\chi_b$ is related to $\chi$ by the fact that, for every $\Gamma$-constructible set $C$ in $\R^n$, 
$$\chi_b(C)=\chi(C\cap [-\gamma,\gamma]^n)$$
for all sufficiently large $\gamma\in \Gamma$ (this formula is also valid when $\Gamma=\{0\}$ and $\gamma=0$). 

Combining these Euler characteristics with the grading on $\mathbf{K}^{\dim}_{\Q}(\Gamma)$, we obtain a morphism of graded rings 
\begin{equation}\label{eq:DOAG}	 \mathbf{K}^{\dim}_{\Q}(\Gamma)\to \Z[u,v]/(uv),\,[C]_n\mapsto \chi(C)u^n+\chi_b(C)v^n.
\end{equation}

\subsection{A Brianchon-Gram decomposition theorem for polyhedra}\label{ss:BriGra}

 A {\em polyhedron} in $\R^n$ is a finite intersection of affine half-spaces, without any rationality conditions on the equations. A {\em polytope} is a bounded polyhedron. For every polyhedron $P$ in $\R^n$, we denote by $\mathbf{1}_P\colon \R^n\to \Z$ its indicator function, which maps points of $P$ to $1$ and all other points of $\R^n$ to $0$.
 
 If $P$ is a polyhedron in $\R^n$ and $F$ is a face of $P$, then we denote by $T_FP$ the tangent cone of $P$ along $F$:
 $$T_FP=\{f+t(p-f)\in \R^n\,|\,f\in F,\,t\in \R_{\geq 0},\,p\in P\}.$$
 The classical Brianchon-Gram decomposition theorem expresses a polytope in terms of its tangent cones. It was originally stated as a relation between the angles of the tangent cones;  we will need a finer statement that can be found, for instance, in  Section 1.1 of \cite{haase}.
 
 \begin{theorem}[Brianchon-Gram decomposition theorem for polytopes]\label{theo:BriGra}
For every polytope $P$ in $\R^n$, we have 
$$\mathbf{1}_{P}=\sum_{F\preccurlyeq P}(-1)^{\dim(F)}\mathbf{1}_{T_FP}$$
where the sum is taken over all the faces of $P$.
 \end{theorem}

 Theorem \ref{theo:BriGra} is false as stated for unbounded polyhedra:  $\R_{\geq 0}$ and $\R$ are the simplest counterexamples. We will prove an adapted statement for arbitrary polyhedra in Theorem \ref{theo:ubBriGra}. A similar formula already appeared as Theorem 4.1 in \cite{DeFi}, but that formulation requires the assumption that $P$ does not contain any lines, which is missing in \cite{DeFi}. 
  For a proof, the authors of \cite{DeFi} refer to \cite{McMu}, but Theorem \ref{theo:ubBriGra} does not follow directly from any statement in that paper (although it contains closely related results). Rather than digging out a proof from the arguments in \cite{McMu}  (which, at its turn, uses ideas from \cite{Shep}), we have chosen to present a direct proof here. We start by recalling some terminology.

 	Let $P$ be a non-empty polyhedron in $\R^n$. The {\em recession cone} of $P$ is defined as
 $$\mathrm{rec}(P)=\{v\in \R^n\,|\,p+v\in P\mbox{ for all }p\in P\}.$$
 Equivalently, a vector $v\in \R^n$ lies in $\mathrm{rec}(P)$ if and only if there exists a point $p$ in $P$ such that $p+tv\in P$ for all $t\in \R_{\geq 0}$. Thus, the recession cone of $P$ collects the unbounded directions of $P$.   It is $\{0\}$ if and only if $P$ is bounded, and it is strictly convex if and only if $P$ does not contain a line (which is also equivalent to the condition that $P$ has a vertex). 
 
 We denote by $\mathrm{Lin}(P)$ the maximal linear subspace of $\mathrm{rec}(P)$, and by $\ell(P)$ its dimension. The vector space $\mathrm{Lin}(P)$ acts on all faces of $P$ by translation. We say that a face $F$ of $P$ is {\em relatively bounded} if $\mathrm{rec}(F)=\mathrm{Lin}(P)$; equivalently, if the image of $F$ in the quotient space  $\R^n/\mathrm{Lin}(P)$ is bounded.

 A {\em polyhedral subdivision} of $P$ is a finite set $\mathscr{P}$ of polyhedra in $\R^n$ such that $P$ is the union of the elements in $\mathscr{P}$, any face of an element of $\mathscr{P}$ again lies in $\mathscr{P}$, and the intersection of two elements of $\mathscr{P}$ is a common face of both.

 Finally, let $F$ be a non-empty face of $P$, and let $x$ be a point in $\R^n\setminus P$. Let $y$ be a point in the relative interior of $F$. It is easy to see that the line segment $[xy]$ intersects $P$ only at the point $y$ if and only if $x$ does not lie in the tangent cone $T_FP$. In particular, this property only depends on $F$, and not on the choice of the  point $y$. If it is satisfied, then we say that $F$ is {\em visible} from $x$.

\begin{prop}\label{prop:contract}
	Let $P$ be a polyhedron in $\R^n$ that does not contain any lines. Let $\mathscr{P}$ be a polyhedral subdivision of $P$. Then the union of the bounded faces of $\mathscr{P}$ is contractible.	
\end{prop}	
\begin{proof}
	Our argument generalizes the proof of Lemma 4.5 in \cite{Hirai}.
		Let $U_b$ the union of the bounded faces of $\mathscr{P}$. It suffices to construct a continuous retraction $r\colon P\to U_b$. By convexity of $P$, the subspace $U_b$ is then a strong deformation retract of $P$, and the result follows from the contractibility of $P$.
	
	When $Q$ runs over the non-empty polyhedra in $\mathscr{P}$, the recession cones $\mathrm{rec}(Q)$ fit together in a fan $\mathrm{rec}(\mathscr{P})$ of strictly convex polyhedral cones  that subdivides $\mathrm{rec}(P)$, by Theorem 3.4 in \cite{BuSo}. This fan is called the recession fan of $\mathscr{P}$. We choose generators $u_1,\ldots,u_\ell$ for the different rays of $\mathrm{rec}(\mathscr{P})$. 
	
	Let $i\in \{1,\ldots,\ell\}$, let $Q$ be a polyhedron in $\mathscr{P}$, and let $q$ be a point in $Q$. If $u_i$ does not lie in the recession cone of $Q$, then we set $t_{\min}(Q,q,u_i)=0$. Otherwise, we set $$t_{\min}(Q,q,u_i)=\min \{t \in \R\,|\,q+tu_i\in Q\}.$$
		The value $t_{\min}(Q,q,u_i)$ is well-defined, because $Q$ does not contain any lines.
	We use it to define the continuous map   
	 $$r_{i,Q}\colon Q\to Q,\,q\mapsto q+t_{\min}(Q,q,u_i)u_i.$$
		If $F$ is a face of $Q$, and $u_i$ lies in the recession cone of $Q$, then $u_i$ lies in the recession cone of $F$ if and only if it lies in the translation space of the affine span of $F$ in $\R^n$. This implies that, for every point $q$ in $F$, we have 
		$t_{\min}(Q,q,u_i)=t_{\min}(F,q,u_i)$. Consequently,  
	 $r_{i,F}$ is the restriction of $r_{i,Q}$ to $F$, so that the maps $r_{i,Q}$ glue to a continuous map 
		$r_i\colon  P\to P$. We have $r_i\circ r_i=r_i$, and the image of $r_i$ is the union of the polyhedra in $\mathscr{P}$ whose recession cones do not contain $u_i$.
		
		We further set 
		$$r=r_{\ell}\circ \ldots \circ r_1\colon P\to P.$$
		By construction, $r$ is the identity on $U_b$. The image of $r$ is contained in $U_b$, because a non-empty polyhedron $Q$ in $\mathscr{P}$ is bounded if and only if its recession cone does not contain any of the generators $u_1,\ldots,u_\ell$. This shows that $U_b$ is a retract of $P$.
\end{proof}

\begin{corollary}\label{coro:contract}
	Let $P$ be a non-empty polyhedron in $\R^n$.
	\begin{enumerate}
		\item \label{it:unionb} The union $U_b$ of the relatively bounded faces of $P$ is contractible, and it has Euler characteristic $\chi(U_b)=(-1)^{\ell(P)}$.  
		\item \label{it:unionvis} For every point $x\in \R^n\setminus P$, the union $U_v$ of the relatively bounded faces of $P$ visible from $x$ is contractible, and it has Euler characteristic $\chi(U_v)=(-1)^{\ell(P)}$. 
	\end{enumerate}
\end{corollary}
\begin{proof}
	\eqref{it:unionb} The projection $\pi\colon \R^n\to \R^n/\mathrm{Lin}(P)$ induces a bijective correspondence between the relatively bounded faces of $P$ and the bounded faces of $\pi(P)$, and $\pi$ is a trivial fibration with contractible fiber $\mathrm{Lin}(P)$, which has Euler characteristic 
	$\chi(\mathrm{Lin}(P))=(-1)^{\ell(P)}$.  
		Therefore, we may replace $P$ by the polyhedron $\pi(P)$ in $\R^n/\mathrm{Lin}(P)$, and thus reduce to the case where $\ell(P)=0$; that is, where $P$ does not contain any lines. Now the result follows by   
	applying Proposition \ref{prop:contract} to the trivial polyhedral subdivision of $P$, consisting of all the faces of $P$.  
	
	\eqref{it:unionvis} We can reduce to the case where $P$ does not contain any lines as in the proof of \eqref{it:unionb}; note that a face $F$ of $P$ is visible from $x$ if and only if $\pi(F)$ is visible from $\pi(x)$ as a face of $\pi(P)$.
		 We may further assume that $x$ is the origin $0$ of $\R^n$. We denote by $P_b$ be the bounded part of $P$; that is, the convex hull of its vertices. We further write $\mathrm{cone}(P_b)$ for the cone over $P_b$ with vertex $0$; this is a polyhedral cone in $\R^n$. Then we define a polyhedron $Q$ as the Minkowski sum of $P$ and $\mathrm{cone}(P_b)$:
	$$Q=P+\mathrm{cone}(P_b).$$
	We will show that the bounded faces of $Q$ are precisely the bounded faces of $P$ that are visible from $0$, so that the result follows from \eqref{it:unionb}.
	
	Every face of $Q$ is the sum of a face of $P$ and a face of $\mathrm{cone}(P_b)$, so that every bounded face $F$ of $Q$ is also a face of $P$. Moreover, $F$ is visible from $0$ as a face of $Q$, and, {\em a fortiori}, also as a face of $P\subset Q$: if $L$ is a line through $0$ and a point $y$ in $F$, then $L\cap F$ is a face of $L\cap Q$, which is a closed half-line pointing away from $0$ by the construction of $Q$. Therefore, $L\cap F=\{y\}$, and the open segment $(0y)$ does not intersect $Q$. Conversely, if $F$ is a face of $P$ that is visible from $0$, then $P$ has a supporting hyperplane $H$ that contains $F$ and separates $0$ and $P$, because $0$ does not lie in the tangent cone $T_FP$. Then $H$ is still a supporting hyperplane of $Q$, and $F$ is a face of $Q\cap H=P\cap H$, so that $F$ is a face of $Q$. 
	\end{proof}

The following theorem generalizes the Brianchon-Gram theorem to arbitrary  polyhedra.

\begin{theorem}\label{theo:ubBriGra}
Let $P$ be a polyhedron in $\R^n$. If we denote by $B(P)$ the set of non-empty relatively bounded faces of $P$, then 
\begin{equation}\label{eq:ubBriGra}
			\mathbf{1}_{P}=\sum_{F\in B(P)}(-1)^{\dim(F)+\ell(P)}\mathbf{1}_{T_FP}.
	\end{equation}	
\end{theorem}
\begin{proof}
	 We may assume that $P$ is non-empty; otherwise, the theorem is trivial.
 Let $x$ be a point in $\R^n$.	We denote by $U_b$ the union of the relatively bounded faces of $P$, and by $U_v$ the union of the relatively bounded faces of $P$ that are visible from $x$.
	
 If $x\in P$, then $x$ lies in $T_FP$ for every face $F\in B(P)$. Evaluating the right hand side of \eqref{eq:ubBriGra} at $x$ computes 
 $(-1)^{\ell(P)}\chi(U_b)$, which equals $1$ by Corollary \ref{coro:contract}. 
If $x\notin P$, then $x\in T_FP$ if and only if the face $F$ is not visible from $x$. Therefore, evaluating the right hand side of \eqref{eq:char} at $x$ yields $(-1)^{\ell(P)}(\chi(U_b)-\chi(U_v))$, which vanishes by Corollary \ref{coro:contract}. 
 Consequently, \eqref{eq:ubBriGra} is satisfied at all  points $x$ in $\R^n$.   	
\end{proof}


\subsection{Generators for $\mathbf{K}^{\dim}(\Gamma)$}\label{ss:generators}
We will now use Theorem \ref{theo:ubBriGra} to prove that the ring $\mathbf{K}^{\dim}(\Gamma)$ is generated by $[\R_{\geq 0}]_1$ and $[\R_{>0}]_1$ if $\Gamma$ is divisible.

	\begin{lemma}\label{lemm:poly}
	Every element of $\mathbf{K}^{\dim}(\Gamma)$ can be written as a $\Z$-linear combination of classes of $\Gamma$-rational polyhedra. 
\end{lemma}
\begin{proof}
	Let $C$ be a $\Gamma$-constructible subset of $\R^n$. By definition, $C$ is a finite Boolean combination 
	of $\Gamma$-rational polyhedra in $\R^n$. Putting this Boolean combination into algebraic normal form, we can write 
	$$C=P_1 \oplus \ldots \oplus P_r$$
	where $P_1,\ldots,P_r$ are polyhedra in $\R^n$ and $\oplus$ denotes the exclusive disjunction. In other words, $\Gamma$ consists of the points $x$ in $\R^n$ such that the number of indices $i$ in  $\{1,\ldots,r\}$ for which $x\in P_i$ is odd. By an inclusion-exclusion argument, we find that 
	$$[C]_n=\sum_{J}(-1)^{|J|-1}[P_J]_n$$
	where $J$ runs through the non-empty subsets of $\{1,\ldots,r\}$ and $P_J$ is the intersection of the polyhedra $P_j$ with $j\in J$. In particular, $[C]_n$ is a $\Z$-linear combination of classes of polyhedra.
\end{proof}

\begin{lemma}\label{lemm:cone}
	Let $C$ be a rational polyhedral cone in $\R^n$. Then the class of $C$ in $\mathbf{K}^{\dim}(\Gamma)$ is given by 
	$$[C]_n=\left\{
	\begin{array}{ll} 	
		[\R_{\geq 0}]^n_1+(-1)^{n-\dim(C)}[\R_{>0}]^n_1 & \mbox{if }C\mbox{ is a linear subspace of }\R^n,
		\\[1.5ex]  [\R_{\geq 0}]_1^n & \mbox{otherwise.}
		\end{array} 
		\right.$$
\end{lemma}
\begin{proof}
	For every $d\in \{0,\ldots,n\}$, we have 
	$$	[\R^d\times \{0\}^{n-d}]_n=\sigma^{n-d}([\R_{\geq 0}]_1+[\R_{> 0}]_1)^{d}.$$ 
Since  $\sigma=[\R_{\geq 0}]_1-[\R_{>0}]_1$, and $[\R_{\geq 0}]_1\cdot [\R_{>0}]_1=0$ by Example \ref{exam:prodzero}, this yields the desired expression when $C$ is a linear subspace of $\R^n$. 

Therefore, we may assume that $C\neq \mathrm{Lin}(C)$. This implies that  the intersection of $C$ with the unit sphere $S^n\subset \R^n$ is contractible.  
		We can subdivide $C$ into simple cones. By inclusion-exclusion and Euler's formula 
		for simplicial subdivisions of $C\cap S^n$, this reduces the problem of proving that $[C]_n=[\R_{\geq 0}]_1^n$ to the case where $C$ is itself simple.
		Then $C$ is integrally isomorphic over $\{0\}$ to $$(\R_{\geq 0})^{\dim(C)} \times \{0\}^{n-\dim(C)},$$ so that $$[C]_n=\sigma^{n-\dim(C)}[\R_{\geq 0}]_1^{\dim(C)}.$$ Since $\dim(C)>0$, and $[\R_{\geq 0}]_1\cdot [\R_{>0}]_1=0$ by Example \ref{exam:prodzero}, we find that $[C]_n=[\R_{\geq 0}]^n_1$.
\end{proof}

\begin{lemma}\label{lemm:brianchon-gram}
If $\Gamma$ is divisible, then every element in $\mathbf{K}^{\dim}(\Gamma)$ is a $\Z$-linear combination of classes of rational polyhedral cones.	
\end{lemma}
\begin{proof}	
	By Lemma \ref{lemm:poly}, it suffices to prove this for the class of an arbitrary $\Gamma$-rational polyhedron $P$ in $\R^n$. 
	For every non-empty face $F$ of $P$, the tangent cone $T_FP$ of $P$ along $F$ is a $\Gamma$-rational polyhedron: if we write $P$ as an intersection of finitely many $\Gamma$-rational half-spaces in $\R^n$, then $T_FP$ is the intersection of the subset of half-spaces that contain $F$ in their boundaries. 
		 Since $\Gamma$ is divisible, the minimal face of $T_FP$ contains a $\Gamma$-rational point.
	Consequently, up to translation by an element in $\Gamma^n$, the tangent cone $T_FP$  is a rational polyhedral cone in $\R^n$. 
	
	We say that a function $f\colon \R^n\to \Z$ is $\Gamma$-constructible if it takes only finitely many values and its fibers are $\Gamma$-constructible sets in $\R^n$. We denote by $\mathcal{C}_{\Gamma}(\R^n)$ the abelian group of $\Gamma$-constructible functions on $\R^n$, where the group operation is pointwise addition, and we form the graded group  $$\mathcal{C}^{\dim}_{\Gamma}=\bigoplus_{n\geq 0}\mathcal{C}_{\Gamma}(\R^n).$$
		The {\em integration morphism}
	\begin{equation}\label{eq:intGamma}
		\mathcal{C}^{\dim}_{\Gamma}\to \mathbf{K}^{\dim}(\Gamma)
	\end{equation}
	 maps each  $\Gamma$-constructible function $f$ on $\R^n$ to 
	$$\sum_{m\in \Z}m[f^{-1}(m)]_n \in \mathbf{K}(\Gamma[n]),$$
	for every $n\geq 0$. This is a morphism of graded groups, by the scissor relations in $\mathbf{K}^{\dim}(\Gamma)$.
	
	The Brianchon-Gram theorem for polyhedra (Theorem \ref{theo:ubBriGra}) tells us that 
		\begin{equation}\label{eq:char}
		\mathbf{1}_{P}=\sum_{F\in B(P)}(-1)^{\dim(F)+\ell(P)}\mathbf{1}_{T_FP}\end{equation}
		in $\mathcal{C}_{\Gamma}(\R^n)$, 
	where $B(P)$ is the set of non-empty relatively bounded faces of $P$. Applying the integration morphism \eqref{eq:intGamma}, we obtain the equality 
$$		[P]_n=\sum_{F\in B(P)}(-1)^{\dim(F)+\ell(P)}[T_F P]_n$$
	in $\mathbf{K}^{\dim}(\Gamma)$,  which expresses $[P]_n$ as a $\Z$-linear combination of classes of rational polyhedral cones.   
\end{proof}

\begin{prop}\label{prop:inR}
	If $\Gamma$ is divisible, then the ring $\mathbf{K}^{\dim}(\Gamma)$ is generated by the elements $[\R_{\geq 0}]_1$ and $[\R_{> 0}]_1$.
\end{prop}
\begin{proof}
	By Lemma \ref{lemm:cone}, the class of any rational polyhedral cone lies in the subring of $\mathbf{K}^{\dim}(\Gamma)$ generated by $[\R_{\geq 0}]_1$ and $[\R_{> 0}]_1$.  Now the result follows from Lemma \ref{lemm:brianchon-gram}.
\end{proof}

\subsection{Explicit description of $\mathbf{K}^{\dim}(\Gamma)$.}\label{ss:explicit}
\begin{theorem}\label{theo:explicit}
Assume that $\Gamma$ is divisible.
\begin{enumerate}
	\item \label{it:explicit} There exists a unique isomorphism of graded rings
$$\mathbb{Z}[u,v]/(uv)\to \mathbf{K}^{\dim}(\Gamma)$$
that maps $u$ to $-[\mathbb{R}_{> 0}]_{1}$ and $v$ to $[\R_{\geq 0}]_1$.
 The inverse isomorphism 
 $$\mathbf{K}^{\dim}(\Gamma)\to \mathbb{Z}[u,v]/(uv)$$
 sends each element $\alpha$ of $\mathbf{K}^{\dim}(\Gamma)$ to 
 $$\sum_{n\geq 0}\left(\chi(\alpha_n)u^n+\chi_b(\alpha_n)v^n\right),$$ where $\alpha_n$ denotes the degree $n$ part of $\alpha$.

    \item \label{it:compariso} The morphism $$\mathbf{K}^{\dim}(\Gamma)\to \mathbf{K}^{\dim}_{\Q}(\Gamma)$$ in \eqref{eq:compar} is an isomorphism of graded rings.	
 \end{enumerate}
\end{theorem}	
\begin{proof}
Consider the unique morphism of graded rings  
$$\mathbb{Z}[u,v]\to \mathbf{K}^{\dim}(\Gamma)$$
that maps $u$ to $-[\mathbb{R}_{> 0}]_{1}$ and $v$ to $[\R_{\geq 0}]_1$. 
 This morphism is surjective, by Proposition \ref{prop:inR}, and 
 Example \ref{exam:prodzero} shows that it factors through a morphism 
\begin{equation}\label{eq:explicit}
	\mathbb{Z}[u,v]/(uv)\to \mathbf{K}^{\dim}(\Gamma).
\end{equation}	
 We now have a chain of morphisms of graded rings 
\begin{center}\begin{tikzcd}[column sep=large]	\mathbb{Z}[u,v]/(uv) \arrow[r,"\eqref{eq:explicit}"] & \mathbf{K}^{\dim}(\Gamma) \arrow[r,"\eqref{eq:compar}"] & \mathbf{K}^{\dim}_{\Q}(\Gamma) \arrow[r,"\eqref{eq:DOAG}"] & \Z[u,v]/(uv).
		\end{tikzcd}\end{center}
 The composite of the chain is the identity on $\Z[u,v]/(uv)$, because   $$\begin{array}{l|l}
  	 \chi(\R_{\geq 0})=0 &\chi_b(\R_{\geq 0})=1 \rule{0pt}{2.6ex} \rule[-0.9ex]{0pt}{0pt}
\\[1.5pt]  	\chi(\R_{> 0})=\chi(\R_{\geq 0})-\chi(\{0\})=-1 & \chi_b(\R_{> 0})=\chi_b(\R_{\geq 0})-\chi_b(\{0\})=0 
  	\end{array}$$ 
  	  	Since the morphisms \eqref{eq:explicit} and \eqref{eq:compar} are surjective,
  	 it follows that all the morphisms in the chain are isomorphisms, and that the composite of \eqref{eq:compar} and \eqref{eq:DOAG} is inverse to \eqref{eq:explicit}.
 \end{proof}

 \begin{corollary}
	If $\Gamma$ is divisible, then the ring $\mathbf{K}^{\dim}(\Gamma)$	does not depend on $\Gamma$: the ``extension of scalars'' morphism 
	$\mathbf{K}^{\dim}(\Gamma)\to \mathbf{K}^{\dim}(\R)$ is an isomorphism.
\end{corollary}
\begin{proof}
	Immediate from Theorem \ref{theo:explicit}.
\end{proof}

 \begin{corollary}\label{coro:classpoly}
	For every non-empty $\Gamma$-rational polyhedron $P$ in $\R^n$, the class $[P]_n$ in $\mathbf{K}(\Gamma[n])$ is given by 
	$$[P]_n=\left\{\begin{array}{ll}
		 [\R_{\geq 0}]^n+(-1)^n[\R_{> 0}]^n& \mbox{ if }P\mbox{ is bounded;}
		\\[1.5ex] [\R_{\geq 0}]^n+[\R_{> 0}]^n & \mbox{if }P=\R^n;
		\\[1.5ex] [\R_{\geq 0}]^n_1 & \mbox{otherwise.} 
		\end{array} \right.$$
\end{corollary}
\begin{proof}
	We always have $\chi_b(P)=1$. 
	If $P$ is bounded, then $\chi(P)=1$. Moreover, $\chi(\R^n)=(-1)^n$. Finally, if $P$ is unbounded but different from $\R^n$, then the one-point compactification of $P$ is contractible, so that  $\chi(P)=0$.  
	Now the result follows from Theorem \ref{theo:explicit}.
\end{proof}

 The condition that $\Gamma$ is divisible is essential in Theorem \ref{theo:explicit}: we will prove in Corollary \ref{cor:nondiv} below that the theorem is false whenever $\Gamma$ is not divisible.

\subsection{A question by Hrushovski and Kazhdan}\label{ss:question}
A slightly coarser version of the ring $\mathbf{K}^{\dim}_{\Q}(\Gamma)$ has been studied in the context of model theory of divisible ordered abelian groups. 
 For all integers $n\geq 0$, multiplication by $\sigma=[\{0\}]_1$ defines a morphism of abelian groups 
\begin{equation}\label{eq:dimensionup}
	\mathbf{K}_{\Q}(\Gamma[n])\to \mathbf{K}_{\Q}(\Gamma[n+1]),\,\alpha\mapsto \sigma \alpha.
\end{equation} 
We define $\mathbf{K}_{\Q}(\Gamma)$ as the colimit of the groups $\mathbf{K}_{\Q}(\Gamma[n])$ over $n\geq 0$. The Cartesian product of $\Gamma$-constructible sets induces a ring structure on $\mathbf{K}_{\Q}(\Gamma)$. Equivalently, we can define $\mathbf{K}_{\Q}(\Gamma)$ as the quotient of $\mathbf{K}^{\dim}_{\Q}(\Gamma)$ by the ideal $(\sigma-1)$. The effect of passing from $\mathbf{K}^{\dim}_{\Q}(\Gamma)$ to $\mathbf{K}_{\Q}(\Gamma)$ is that we forget the dimensions of the ambient vector spaces of $\Gamma$-constructible sets. Analogously, one defines the ring $\mathbf{K}(\Gamma)$ as the colimit of the groups $\mathbf{K}(\Gamma[n])$ or as the quotient of $\mathbf{K}^{\dim}(\Gamma)$ by the ideal $(\sigma-1)$.

In the model-theoretic literature, the ring $\mathbf{K}_{\Q}(\Gamma)$ is known as $\mathbf{K}(\mathrm{DOAG}_{\Gamma})$, the Grothendieck ring of the theory of divisible ordered abelian groups with parameters in $\Gamma$. It is completely understood when $\Gamma$ is divisible.

\begin{theorem}\label{theo:DOAG}
	Assume that $\Gamma$ is divisible, and consider the ring morphism 
	$$\mathbf{K}_{\Q}(\Gamma)\to \Z\times \Z$$ that sends the class of a $\Gamma$-constructible set $C$ to  $(\chi(C),\chi_b(C))$. This is an isomorphism.	
\end{theorem} 
\begin{proof}
This has been proved independently by several authors \cite{HK,FK,Ma}. 
	\end{proof}	

It is easy to enhance this result to a graded version: if $\Gamma$ is divisible, then 
the ring morphism 
$$\mathbf{K}^{\dim}_{\Q}(\Gamma)\to \Z[u,v]/(uv),\,[C]_n\mapsto \chi(C)u^n+\chi_b(C)v^n$$ is an isomorphism of graded rings. This property is contained in Theorem \ref{theo:explicit}, and already implicit in \cite{HK}. 

In Question 9.9 of \cite{HK}, Hrushovski and Kazhdan ask\footnote{Beware of a conflict of notation with \cite{HK}: they use $A$ to denote our parameter group $\Gamma$, and $\Gamma$ to denote a divisible ordered abelian group that contains $A$ as a subgroup. Their question is still broader than what is stated here, because they do not assume that $A$ is a subgroup of $\R$.} whether for {\em all} subgroups $\Gamma$ of $\R$, the natural morphism 
$$\mathbf{K}^{\dim}(\Gamma)/\mathrm{Ann}(\sigma)\to \mathbf{K}^{\dim}_{\Q}(\Gamma)/\mathrm{Ann}(\sigma)$$
is an isomorphism, where $\mathrm{Ann}(\sigma)$ denotes the ideal of elements killed by some power of $\sigma$. 
 Theorem \ref{theo:explicit} implies that, if $\Gamma$ is divisible, then $\mathrm{Ann}(\sigma)=\{0\}$ and 
\begin{equation}\label{eq:compar2}
	\mathbf{K}^{\dim}(\Gamma)\to \mathbf{K}^{\dim}_{\Q}(\Gamma)\end{equation} 
	is an isomorphism. In particular, we obtain a positive answer to Hrushovski and Kazhdan's question in the case where $\Gamma$ is divisible. A weaker statement was proved by Hrushovski and Kazhdan in Theorem 3.12 of \cite{HK2}: there, they showed that the morphism 
	$$	\mathbf{K}(\Gamma)\otimes_{\Z}\Q \to \mathbf{K}_{\Q}(\Gamma)\otimes_{\Z}\Q $$
	is an isomorphism if $\Gamma$ is divisible (that is, after forgetting ambient dimensions and tensoring with $\Q$ -- beware that they use $\mathbf{K}_{\Q}(\Gamma)$ to denote $\mathbf{K}(\Gamma)\otimes_{\Z}\Q $).

	The question becomes much more subtle when $\Gamma$ is not divisible: if $a$ is an element of $\Gamma\otimes_{\Z}\Q$, then $[\{a\}]_1-\sigma$ lies in the kernel of \eqref{eq:compar2}. It is proved in Lemma 9.7 of \cite{HK} that $\sigma([\{a\}]_1-\sigma)=0$ in $\mathbf{K}^{\dim}(\Gamma)$ if $2a\in \Gamma$; then $[\{a\}]_1-\sigma\in \mathrm{Ann}(\sigma)$. The $\Gamma$-rational  polytopes $\{0\}$ and $\{a\}$ in $\R$ are clearly not integrally isomorphic over $\Gamma$ if $a\notin \Gamma$, but it is less obvious that they define different classes in $\mathbf{K}^{\dim}(\Gamma)$. Since no argument is given in \cite{HK}, we provide a proof here.
	
\begin{prop}\label{prop:distinct}
	Let $a$ be an element in $\Gamma\otimes_{\Z}\Q$. Then $[\{a\}]_1= \sigma$ in $\mathbf{K}^{\dim}(\Gamma)$ if and only if $a\in \Gamma$. 
\end{prop}	 
\begin{proof}
	If $a$ lies in $\Gamma$, then $\{0\}$ and $\{a\}$ are integrally isomorphic over $\Gamma$ {\em via} translation by $a$. Therefore, we may assume that $a\notin \Gamma$; then we must show that $[\{a\}]_1\neq  \sigma$ in $\mathbf{K}(\Gamma[1])$.
		
		To deduce this result, we propose a new additive invariant of $\Gamma$-constructible sets in $\R$, which is specific to the $1$-dimensional case. It is easy to see that the $\Gamma$-constructible sets in $\R$ are precisely the finite disjoint unions of points in $\Gamma\otimes_{\Z}\Q$ and open intervals with end-points in $(\Gamma\otimes_{\Z}\Q)\cup \{\pm \infty\}$: it suffices to check that this class of subsets of $\R$ is closed under finite Boolean combinations. 
		
	Let $C$ be a $\Gamma$-constructible set in $\R$, and let $x$ be an element of $\Gamma$. We define the {\em weight} $w_{C}(x)$ of $x$ with respect to $C$ in the following way:
	\begin{itemize}
	\item $w_C(x)=2$ if $x$ is an isolated point of $C$;
	\item $w_C(x)=1$ if exactly one of $[x,x+\varepsilon)$ and $(x-\varepsilon,x]$ is contained in $C$, for all sufficiently small $\varepsilon >0$;
	\item $w_C(x)=0$ if $x$ lies in the interior of $C$ or $\R\setminus C$;
	\item $w_C(x)=-1$ if $x\notin C$ and exactly one of $(x,x+\varepsilon)$ and $(x-\varepsilon,x)$ is contained in $C$, for all sufficiently small $\varepsilon >0$;
		\item $w_C(x)=-2$ if $x$ is an isolated point in $\R\setminus C$.
	\end{itemize}
	Note that $w_C(x)=-w_{\R\setminus C}(x)$. Furthermore,   
	  $w_C(x)=0$ for all but finitely many points $x$ in $\Gamma$, so that we can  define 
	$$\chi_{\Gamma}(C)=\sum_{x\in \Gamma}w_C(x).$$
	It is obvious that $\chi_{\Gamma}(C)$ is invariant under integral isomorphisms over $\Gamma$. We will now  prove that it is additive on disjoint unions, so that it satisfies the scissor relations in $\mathbf{K}(\Gamma[1])$.
	
	Let $C$ and $D$ be disjoint $\Gamma$-constructible subsets of $\R$, and set $E=C\cup D$. We will prove that 
	$w_E(x)=w_C(x)+w_D(x)$ for every point $x$ in $\Gamma$.   

\subsubsection*{Case 1: $D$ consists of a unique point $d$}   Clearly, $w_E(x)=w_C(x)+w_D(x)$ whenever $x\neq d$, so that we may assume that $x=d\in \Gamma$. Then $w_D(x)=2$, and a simple case-by-case analysis on the values of $w_C(x)$ shows that $w_E(x)=w_C(x)+2$ (note that $w_C(x)\leq 0$ because $x\notin C$).

\subsubsection*{Case 2: $D$ is a non-empty bounded open interval $(d_1,d_2)$}
It is clear that $w_E(x)=w_C(x)+w_D(x)$ whenever $x\notin \{d_1,d_2\}$. By symmetry, we may assume that $x=d_1\in \Gamma$.  Then $w_D(x)=-1$ and $w_C(x)\geq -1$. Another case-by-case analysis again confirms that 
	$w_E(x)=w_C(x)-1$.
	
\subsubsection*{Case 3: $D$ is a half-bounded open interval} By symmetry, we may assume that $D=(d,+\infty)$ for some $d\in \Gamma\otimes_{\Z} \Q$. As in the previous cases, we may further assume that $x=d\in \Gamma$; then the argument is the same as in Case 2.

\subsubsection*{General case} We can write any $\Gamma$-constructible set $D$ in $\R$ as a finite disjoint union of points and open intervals, so that the previous cases imply that $w_E(x)=w_C(x)+w_D(x)$ for every point $x$ in $\Gamma$.

\medskip This additivity property implies, in particular, that $$\chi_{\Gamma}(E)=\chi_{\Gamma}(C)+\chi_{\Gamma}(D).$$
It follows that there exists a unique group morphism 
$$\chi_{\Gamma}\colon \mathbf{K}(\Gamma[1])\to \Z$$ that maps 
$[C]_1$ to $\chi_{\Gamma}(C)$ for every $\Gamma$-constructible set $C$ in $\R$. This morphism distinguishes the classes $[\{a\}]_1$ and $\sigma$, because $\chi_{\Gamma}(\{a\})=0$ whereas $\chi_{\Gamma}(\sigma)=2$.
	\end{proof}
\begin{corollary}\label{cor:nondiv}
Points (1) and (2) in Theorem \ref{theo:explicit} are both false whenever  $\Gamma$ is not divisible.	
\end{corollary}	
\begin{proof}
	Let $a$ be an element of $\Gamma\otimes_{\Z}\Q$ that does not lie in $\Gamma$.  The classes $[\{a\}]_1$ and $\sigma$ are distinct in $\mathbf{K}^{\dim}(\Gamma)$, by Proposition \ref{prop:distinct}, but they are equal in $\mathbf{K}^{\dim}_{\Q}(\Gamma)$. In particular, they have the same Euler characteristic and bounded Euler characteristic. 
\end{proof}

If $\Gamma$ is divisible, then the invariant $\chi_{\Gamma}$ from the proof of Proposition \ref{prop:distinct} is nothing but $\chi+\chi_b$: by Theorem \ref{theo:explicit}, we only need to observe that $\chi_{\Gamma}$ and $\chi+\chi_b$ take the same values on $\R_{\geq 0}$ and $\R_{>0}$. If $\Gamma$ is not divisible, then $\chi_{\Gamma}$ cannot be expressed in terms of $\chi$ and $\chi_b$ alone.    

\section{Application: the kernel of the motivic volume}\label{sec:appli}
In \cite{HK}, Hrushovski and Kazhdan developed a powerful theory of motivic integration over a non-archimedean valued field $K$ of residue characteristic $0$, based on model theory. One of the central results in this theory is a description of $\mathbf{K}(\VF_K)$, the Grothendieck ring of semi-algebraic sets over $K$, in terms of the Grothendieck ring of polytopes over the value group of $K$ and a generalized Grothendieck ring of varieties over the residue field of $K$. This makes it possible to define and compute geometric invariants of semi-algebraic sets over $K$ from those of algebraic varieties over the residue field. 
 In this section, we will use Theorem \ref{theo:explicit} to simplify the description of $\mathbf{K}(\VF_K)$ when $K$ is algebraically closed, and draw some interesting conclusions. We start by recalling the ingredients of Hrushovski and Kazhdan's result, formulating them in a more geometric fashion.

 Let $K$ be an algebraically closed non-trivially valued field with residue field $k$ of characteristic zero. We denote the ordered value group by $(\Gamma,\leq)$ and the valuation by  
$$v\colon K^{\ast}\to\Gamma.$$  We extend $v$ to $K$ by setting $v(0)=\infty$,  and declare that $\gamma < \infty$ for all $\gamma$ in $\Gamma$. The  ordered group $\Gamma$ is divisible, because $K$ is algebraically closed.

We further denote by $R$ the valuation ring in $K$, and by $k$ its residue field. We assume that $k$ has characteristic zero. We also assume that $\Gamma$ is an ordered subgroup of $\R$, even though this is not required in \cite{HK}, because we want to apply Theorem \ref{theo:explicit}. The most important case for applications to algebraic geometry is $$K=\bigcup_{d>0}k\llpar t^{1/d}\rrpar,$$ the field of Puiseux series over $k$ -- see for instance \cite{thuong,HL,NPS,NP,NiSh,NO}. Then $(\Gamma,\leq)\cong (\Q,\leq)$, so that our assumption is satisfied.

\subsection{The Grothendieck rings of varieties}\label{ss:Kvar}
Let $F$ be a field.  
    For every $n\geq 0$, we denote by $\mathbf{K}(\Var_F^{\leq n})$ the abelian group defined by the following presentation:
\begin{itemize}
	\item generators: isomorphism classes $[X]_n$ of $F$-schemes $X$ of finite type and dimension at most $n$;
	\item relations: whenever $Y$ is a closed subscheme of $X$, we have 
	$$[X]_n=[Y]_n+[X\setminus Y]_n.$$
\end{itemize}
We combine the groups $\mathbf{K}(\Var_F^{\leq n})$ into a graded abelian group
$$\mathbf{K}^{\dim}(\Var_F)=\bigoplus_{n\geq 0}\mathbf{K}(\Var_F^{\leq n}).$$
 This group has a unique graded ring structure such that, for all integers $m,n\geq 0$ and all $F$-schemes $X$ and $X'$ of finite type and dimension at most $m$, resp.~$n$, we have 
$$[X]_m\cdot [X']_{n}=[X\times_F X']_{m+n}.$$
The identity element for the multiplication is $1=[\Spec F]_0$, the class of a point in degree $0$. We call $\mathbf{K}^{\dim}(\Var_F)$ the {\em Grothendieck ring of $F$-varieties graded by dimension}. 
 It is customary to use the symbol $\LL$ to denote $[\A^1_F]_1$, the class of the affine line placed in degree $1$. We further set $\tau=[\Spec F]_1$, the class of a point in degree $1$. Multiplication by $\tau$ increases the grading by $1$, in parallel with the action of $\sigma$ in $\mathbf{K}^{\dim}(\Gamma)$.

We can similarly define another ring by removing the bounds on dimension. We denote by $\mathbf{K}(\Var_F)$ the abelian group defined by the following presentation:
\begin{itemize}
	\item generators: isomorphism classes $[X]$ of $F$-schemes of finite type;
	\item relations: whenever $Y$ is a closed subscheme of $X$, we have 
	$$[X]=[Y]+[X\setminus Y].$$
\end{itemize}
This abelian group again has a unique ring structure satisfying $[X]\cdot [X']=[X\times_F X']$ for all $F$-schemes $X$ and $X'$ of finite type, with identity element $1=[\Spec F]$. The resulting ring is called the {\em Grothendieck ring of $F$-varieties}. The forgetful morphism 
$$\mathbf{K}^{\dim}(\Var_F)\to \mathbf{K}(\Var_F),\,[X]_n\mapsto [X]$$ is surjective, and its kernel is the ideal generated by $\tau-1$. 

The rings $\mathbf{K}^{\dim}(\Var_F)$ and $\mathbf{K}(\Var_F)$ have profound applications in algebraic geometry; see in particular \cite{CLNS,NOa,LS}.

\subsection{The Grothendieck ring of semi-algebraic sets}\label{ss:GrSemAlg}
\begin{definition}\label{defi:semialg}
	Let $n$ be a non-negative integer. A subset $S$ of $K^n$ is called {\em semi-algebraic} if we can write it as a finite Boolean combination of sets of the form 
	$$\{a\in K^n\,|\,v(f(a))\geq v(g(a))\}$$ with $f,g\in K[x_1,\ldots,x_n]$.
\end{definition}

Taking for $g$ the zero polynomial, one sees that every constructible subset of $K^n$ (in the sense of algebraic geometry) is also semi-algebraic, but there are many more semi-algebraic sets. For instance, the unit ball 
$$R=\{a\in K\,|\,v(a)\geq 0\}$$
is semi-algebraic, but not constructible. It is obvious that a Cartesian product of semi-algebraic sets is again semi-algebraic.

\begin{definition}\label{defi:semialgmorph}
	Let $m$ and $n$ be non-negative integers, and let $S$ and $T$ be semi-algebraic subsets of $K^m$ and $K^n$, respectively. A map $S\to T$ is called {\em semi-algebraic} if its graph is semi-algebraic in $K^{m+n}$. 
\end{definition}

 It is straightforward to check that a composite of semi-algebraic maps is semi-algebraic, and that the inverse of a semi-algebraic bijection is semi-algebraic. Using the semi-algebraic maps as morphisms, we obtain the category of semi-algebraic sets over $K$, with semi-algebraic bijections as isomorphisms.
 
 We can now construct the Grothendieck ring of semi-algebraic sets over $K$ in the same way as the Grothendieck ring of varieties $\mathbf{K}(\Var_F)$ from Section \ref{ss:Kvar}.
  We denote by $\mathbf{K}(\VF_K)$ the abelian group defined by the following presentation:
 \begin{itemize} 
 	\item generators: isomorphism classes $[S]$ of semi-algebraic sets $S$ over $K$;
 	\item relations: for every semi-algebraic set $S$ and every semi-algebraic subset $T$ of $S$, we have 
 	$$[S]=[T]+[S\setminus T].$$
 \end{itemize}
  
  There exists a unique ring structure on $\mathbf{K}(\VF_K)$ such that $$[S]\cdot[S']=[S\times S']$$ for all semi-algebraic sets $S$ and $S'$. The identity element for the multiplication is the isomorphism class of a point. We call $\mathbf{K}(\VF_K)$ the Grothendieck ring of semi-algebraic sets over $K$.

\subsection{The isomorphism of Hrushovski and Kazhdan}\label{ss:IsoHK}
There are two natural constructions that relate the Grothendieck rings $\mathbf{K}(\VF_K)$, $\mathbf{K}^{\dim}(\Gamma)$ and $\mathbf{K}^{\dim}(\Var_k)$. First, for every integer $n\geq 0$, we can consider the {\em tropicalization map}
$$\trop\colon (K^{\ast})^n\to \R^n,\,a\mapsto (v(a_1),\ldots,v(a_n))$$
given by coordinatewise valuation. It follows immediately from the definitions that, for every $\Gamma$-constructible subset $C$ in $\R^n$, its inverse image $\trop^{-1}(C)$ is a semi-algebraic subset of $K^n$. The operation $\trop^{-1}$ clearly respects the scissor relations. Moreover, it sends integrally isomorphic $\Gamma$-constructible sets to isomorphic semi-algebraic varieties: an integral automorphism of $\R^n$ over $\Gamma$ corresponds to an automorphism of the algebraic torus $(K^{\ast})^n$ followed by a translation by a $K$-rational point in the torus.  Consequently, 
we obtain a ring morphism 
\begin{equation}\label{eq:tropinv}
\theta_{\trop}\colon \mathbf{K}^{\dim}(\Gamma)\to \mathbf{K}(\VF_K),\,[C]_n\mapsto [\trop^{-1}(C)].	
\end{equation}

Second, for every $R$-scheme $\cY$, we consider the {\em reduction map}
$$\red_{\cY}\colon \cY(R)\to \cY(k)$$ defined by reducing coordinates modulo the maximal ideal $\frak{m}$ in $R$. For every integer $m\geq 0$, every subscheme $\cY$ of $\A^m_R$ and every subscheme $X$ of $\cY\times_R k$, the set  $\red^{-1}_{\cY}(X(k))$ is semi-algebraic in $\A^m_R(K)=K^m$. 

\begin{example}
If $\cY=\A^m_R$ and $X$ is the hypersurface in $\A^m_k$ defined by a polynomial $f$ in $k[x_1,\ldots,x_m]$, then $\red^{-1}_{\cY}(X(k))$ is the set 
 $$\{a\in K^m\,|\,v(a_1),\ldots,v(a_m)\geq 0,\,v(\widetilde{f}(a))>0\}$$
 where $\widetilde{f}$ is any element in $R[x_1,\ldots,x_n]$ whose reduction modulo $\frak{m}$ equals $f$.
\end{example} 
 
   Using the henselian property of $R$, one shows that there exists a unique ring morphism 
 \begin{equation}\label{eq:redinv}
 	\theta_{\red}\colon \mathbf{K}^{\dim}(\Var_k)\to \mathbf{K}(\VF_K)	
 \end{equation}
 that maps $[X]_n$ to $[\red^{-1}_{\cY}(X(k))]$ for every integer $n\geq 0$, every smooth $R$-scheme $\cY$ of relative dimension $n$, every embedding $\cY\to \A^m_R$ and every subscheme $X$ of $\cY\times_R k$.
 
 \begin{example}\label{exam:ball}
	For every integer $n\geq 0$,  
	we consider  
	$$\mathbb{S}^n=\{a\in (K^{\ast})^n\,|\,v(a_1)=\ldots =v(a_n)=0\},$$
	the $n$-dimensional semi-algebraic torus over $K$. We also 
	write 
	$$\begin{array}{rcl}
	\mathbb{B}^n&=&\{a\in K^n\,|\,v(a_1),\ldots,v(a_n)\geq 0\}
	\\ \mathbb{B}^n_o&=&\{a\in K^n\,|\,v(a_1),\ldots,v(a_n)>0\}
	\end{array}$$ for the closed, resp.~open, unit polydisk around the origin of $K^n$. Then we have 
	 	$$\theta_{\trop}([\gamma]_n)=[\mathbb{S}^n],\,\qquad \theta_{\red}([p]_n)=[\mathbb{B}^n_o]$$ for every  $\gamma\in \Gamma^n$ and every $p\in k^n$. Note that these values depend heavily on the integer $n$, which is why we need the grading by dimension on $\mathbf{K}^{\dim}(\Gamma)$ and $\mathbf{K}^{\dim}(\Var_k)$. We further have 
	\begin{equation}\label{eq:ambi}\begin{array}{rclcl}
	\theta_{\trop}([\R_{\geq 0}]_1)&=&[\mathbb{B}^1\setminus \{0\}]&=&\theta_{\red}(\LL-1), 
	\\[ 1.5ex] \theta_{\trop}([\R_{>0}]_1)&=&[\mathbb{B}^1_o\setminus \{0\}]&=&\theta_{\red}(\tau-1).
	 \end{array}\end{equation}
\end{example}

The equalities in \eqref{eq:ambi} show that the morphisms $\theta_{\trop}$ and $\theta_{\red}$ are not orthogonal. Remarkably, Hrushovski and Kazhdan have proved that these are the {\em only} relations between $\theta_{\trop}$ and $\theta_{\red}$, and that the images of $\theta_{\trop}$ and $\theta_{\red}$ generate the whole ring $\mathbf{K}(\VF_K)$. 

\begin{theorem}[Hrushovski-Kazhdan]\label{theo:HKorig}
The ring morphism
$$
\theta\colon \mathbf{K}^{\dim}(\Gamma)\otimes_{\Z} \mathbf{K}^{\dim}(\Var_k)\to \mathbf{K}(\VF_K)$$
induced by $\theta_{\trop}$ and $\theta_{\red}$ is surjective. Its kernel is the ideal generated by the elements 
$$
([\R_{\geq 0}]_1\otimes 1)-(1\otimes (\LL-1)),\qquad ([\R_{>0}]_1\otimes 1) - (1\otimes (\tau-1)).$$ 
 \end{theorem}
 \begin{proof}
 This follows from Theorem 8.8 and Corollary 10.3 in \cite{HK}; we have reformulated some of their constructions in a more geometric language.  In particular, with the notations from \cite{HK}, we have 
 $$!\mathbf{K}(\mathrm{RES[\ast]})=\mathbf{K}(\mathrm{RES[\ast]})=\mathbf{K}^{\dim}(\Var_k)$$	
 	because $\Gamma$ is divisible.
 \end{proof}
 
 Using our description of $\mathbf{K}^{\dim}(\Gamma)$ in Theorem \ref{theo:explicit}, we can further simplify this result in the following way.
 
 \begin{theorem}\label{theo:HKnew}
 	The ring morphism 
 	$$\theta_{\red}\colon \mathbf{K}^{\dim}(\Var_k)\to \mathbf{K}(\VF_K)$$ is surjective, and its kernel is the ideal $\mathcal{I}$  generated by 
 	$(\LL-1)(\tau-1)$.
 \end{theorem}
 \begin{proof}
This follows directly from Theorems \ref{theo:explicit} and \ref{theo:HKorig}. 	
 \end{proof}

With a slight abuse of notation, we continue to write $\theta_{\red}$ for the induced isomorphism $$\mathbf{K}^{\dim}(\Var_k)/\mathcal{I}\to \mathbf{K}(\VF_K).$$

\begin{remark}
In \cite{HK}, Hrushovski and Kazhdan establish an analog of Theorem \ref{theo:HKorig} at the level of Grothendieck {\em semi-rings}, which is a more delicate question. Theorem \ref{theo:HKnew} does not apply in that context, because the description of the Grothendieck ring of polytopes in Theorem \ref{defi:polytope} uses subtraction in an essential way. For most applications in algebraic geometry, the version for Grothendieck rings suffices.
\end{remark}

\subsection{The motivic volume}\label{ss:MotVol}
A key application of Theorem \ref{theo:HKorig} is the construction of the {\em motivic volume}, the invariant denoted by $\mathcal{E}'$ in Theorem 10.5 of \cite{HK}. This is a ring morphism 
$$\Psi\colon \mathbf{K}(\VF_K)\to \mathbf{K}(\Var_k)$$
 that captures the reduction behaviour of semi-algebraic sets over $K$. It is an extension of Denef and Loeser's {\em motivic nearby fiber} \cite{DL-igusa}, a motivic version of the nearby cycles functor in complex geometry (whence our notation $\Psi$); see \S3.7 in \cite{NPS} for a precise comparison. The motivic volume has found applications in singularity theory \cite{thuong,HL,NP}, enumerative geometry \cite{NPS} and birational geometry \cite{NiSh,NOa,NO}. More generally, it provides a way to extend all additive invariants of algebraic $k$-varieties (that is, invariants that can be defined as group morphisms from $\mathbf{K}(\Var_k)$) to additive invariants of semi-algebraic sets over $K$, by composition with $\Psi$. A prominent example of such an invariant is the Hodge-Deligne polynomial.

In terms of the presentation in Theorem \ref{theo:HKnew}, we can describe the morphism $\Psi$ as the composite of the isomorphism 
$$\theta_{\red}^{-1}\colon \mathbf{K}(\VF_K) \to \mathbf{K}^{\dim}(\Var_k)/\mathcal{I}$$ and the forgetful morphism 
$$\mathbf{K}^{\dim}(\Var_k)/\mathcal{I}\to \mathbf{K}(\Var_k)$$
that ignores the grading; note that omitting the grading collapses $\tau-1$, and therefore $\mathcal{I}$, to $0$. The morphism $\Psi$ is uniquely characterized by the property that 
$(\Psi\circ \theta_{\red})([X]_n)=[X]$ for every integer $n\geq 0$ and every $F$-scheme $X$ of finite type and dimension at most $n$. Thanks to Theorem \ref{theo:HKnew}, we also get an elegant description of the kernel of $\Psi$, so that we know precisely which part of the geometry is collapsed in passing from semi-algebraic sets over $K$ to algebraic varieties over $k$, at least at the level of Grothendieck rings.

\begin{theorem}\label{theo:kernel}
The kernel of the motivic volume morphism  
$$\Psi\colon \mathbf{K}(\VF_K)\to \mathbf{K}(\Var_k)$$ is generated by 
$[\mathbb{B}^1]-1$.
\end{theorem}
\begin{proof}	
Immediate from Theorem \ref{theo:HKnew}, since the kernel of the forgetful morphism $\mathbf{K}^{\dim}(\Var_k)\to \mathbf{K}(\Var_k)$ is generated by $\tau-1$.	
\end{proof}

It is instructive to compare Theorems \ref{theo:HKnew} and \ref{theo:kernel}  with analogous results in Ayoub's theory of nearby cycles on the category of motives \cite{Ayoub-nearby}, and his related  theory of motives of rigid analytic $K$-varieties (geometric upgrades of semi-algebraic sets)  \cite{Ayoub-rigid}. When $K$ is the field of Puiseux series over $k$, Ayoub proves\footnote{More precisely, Ayoub works over the Laurent series field  $k\llpar t\rrpar$ and proves that the category is generated by motives with {\em potential} good reduction.} in Theorem 1.3.23 of \cite{Ayoub-rigid} that the category of rigid analytic motives over $K$ is generated by motives of smooth rigid $K$-varieties with good reduction, which is analogous to the surjectivity of the morphism $\theta_{\red}$ in Theorem \ref{theo:HKorig}. Furthermore, homotopy invariance of rigid analytic motives is imposed by formally contracting the closed unit ball $\mathbb{B}^1$ to a point (Definition 1.3.19 in \cite{Ayoub-rigid}), mirroring the description of the kernel of $\Psi$ in Theorem \ref{theo:kernel}.	

	The connections between Ayoub's formalism and Denef and Loeser's motivic nearby fiber, resp.~the motivic volume of Hrushovski and Kazhdan, are explained in depth in \cite{IS1, AIS, IS2}, resp.~\cite{Forey}.

\end{document}